\documentclass{amsart}

\usepackage{amsmath,amssymb,amsthm}

\usepackage[all]{xy}

\usepackage{dsfont}

\usepackage{enumerate}

\usepackage[shortcuts]{extdash}

\newtheorem{thm}{Theorem}[section]
\newtheorem{cor}[thm]{Corollary}
\newtheorem{prop}[thm]{Proposition}
\newtheorem{lem}[thm]{Lemma}
\theoremstyle{definition}
\newtheorem{defn}[thm]{Definition}
\newtheorem{oq}{Open Question}[section]

\theoremstyle{remark}

\renewcommand{\leq}{\leqslant}
\renewcommand{\geq}{\geqslant}
\renewcommand{\nleq}{\nleqslant}

\renewcommand{\epsilon}{\varepsilon}

\newcommand{\uhr}{\upharpoonright}

\newcommand{\converges}{\mathord{\downarrow}}
\newcommand{\diverges}{\mathord{\uparrow}}

\newcommand{\sub}[1]{_{\textup{\tiny{\fontfamily{cmr}\selectfont #1}}}}

\DeclareMathOperator{\dom}{dom}
\DeclareMathOperator{\gra}{graph}

\begin{document}

\title[Dense Computability, Upper Cones, and Minimal Pairs]{Dense
Computability, Upper Cones, \\ and Minimal Pairs}
\author[Astor]{Eric P.\@ Astor}
\address{Google LLC, 111 8th Ave, New York, NY 10011}
\email{eric.astor@gmail.com}

\author[Hirschfeldt]{Denis R.\@ Hirschfeldt}
\address{Department of Mathematics, University of Chicago}
\email{drh@math.uchicago.edu}

\author[Jockusch]{Carl G.\@ Jockusch, Jr.}
\address{Department of Mathematics, University of Illinois, 1409 W. Green St., Urbana, IL 61801}
\email{jockusch@math.uiuc.edu}


\thanks{Hirschfeldt was partially supported by grants DMS-1101458 and
DMS-1600543 
from the National Science Foundation of the United States, and a
Collaboration Grant for Mathematicians from the Simons
Foundation.}


\date{\today}

\begin{abstract}
This paper concerns algorithms that give correct answers with (asymptotic)
density $1$.  A \emph{dense description} of a function $g : \omega \to \omega$
is a partial function $f$ on $\omega$ such that $\{n : f(n) = g(n)\}$ has
density $1$. We define $g$ to be \emph{densely computable} if it has a partial
computable dense description $f$. Several previous authors have studied the
stronger notions of generic computability and coarse computability, which
correspond respectively to requiring in addition that $g$ and $f$ agree on the
domain of $f$, and to requiring that $f$ be total. Strengthening these two
notions, call a function $g$ \emph{effectively densely computable} if it has a
partial computable dense description $f$ such that the domain of $f$ is a
computable set and $f$ and $g$ agree on the domain of $f$. We compare these
notions as well as asymptotic approximations to them that require for each
$\epsilon > 0$ the existence of an appropriate description that is correct on
a set of lower density of at least $1 - \epsilon$. We determine which
implications hold among these various notions of approximate computability and
show that any Boolean combination of these notions is satisfied by a c.e.\@ set
unless it is ruled out by these implications. We define reducibilities
corresponding to dense and effectively dense reducibility and show that their
uniform and nonuniform versions are different. We show that there are natural
embeddings of the Turing degrees into the corresponding degree structures, and
that these embeddings are not surjective and indeed that sufficiently random
sets have quasiminimal degree. We show that nontrivial upper cones in the
generic, dense, and effective dense degrees are of measure $0$ and use
this fact to show that there are minimal pairs in the dense degrees.
\end{abstract}

\maketitle

\section{Introduction}
\label{sec:intro}

Generic computability and coarse computability, which have played
significant roles in several recent papers such
as~\cite{ACDJL,A,CI,DJS,DI,HJKS,HJMS,I1,I2,JS}, both capture the idea
of computing a function on ``almost all'' inputs, where ``almost all''
is defined in terms of asymptotic density. The difference between the
two notions is that generic computability permits errors of omission
(i.e., divergent computations), while coarse computability permits
errors of commission (i.e., incorrect answers). In this paper, we
introduce and begin the study of a notion of asymptotic computability
in which both forms of error may be present, and reintroduce another,
briefly studied in the 1970's but apparently since forgotten, which
requires that our computations emit a signal whenever an error might be
possible. In particular, we study reducibilities and degree structures
arising from these notions. We assume familiarity with basic notions
from computability theory and algorithmic randomness, as can be found
for instance in~\cite{DH}.

We begin by defining the four notions of asymptotic computation we
will study, followed by some historical remarks, definitions of
computability bounds corresponding to these notions, and a brief outline
of the paper.

\begin{defn}
Let $A \subseteq \omega$. The \emph{density of $A$ below $n$}, denoted
by $\rho_n(A)$, is $\frac{|A \uhr n|}{n}$.

The \emph{upper (asymptotic) density} $\overline{\rho}(A)$ of $A$ is
$\limsup_n \rho_n(A)$.

The \emph{lower (asymptotic) density} $\underline{\rho}(A)$ of $A$ is
$\liminf_n \rho_n(A)$.

If $\overline{\rho}(A)=\underline{\rho}(A)$ then we call this quantity
the \emph{(asymptotic) density} of $A$, and denote it by $\rho(A)$.
\end{defn}

As mentioned above, the following two notions have begun to be widely
studied. (We define our notions of asymptotic computability for
functions, but as usual we identify a set with its characteristic
function.)

\begin{defn}
Let $g : \omega \to \omega$. A \emph{partial description} of $g$ is a
partial function $f : \omega \to \omega$ such that $f(n)=g(n)$
whenever $f(n)$ is defined.  A \emph{generic description} of $g$ is a
partial description of $g$ with domain of density $1$. A function is
\emph{generically computable} if it has a partial computable generic
description.

A \emph{coarse description} of a function $g : \omega \to \omega$ is a
(total) function $f : \omega \to \omega$ such that $f(n) = g(n)$ on a
set of density $1$. A function is \emph{coarsely computable} if it has
a computable coarse description.
\end{defn}

When considering sets rather than functions, it is usual to assume
that coarse descriptions are themselves sets, but this assumption
makes no difference in the definition of coarse computability and
related notions. In a few places below, we will mention effective
(Cohen) genericity, and in particular $1$\-/genericity. The collision
in nomenclature between generic computation and effective genericity
is unfortunate, particularly in results that connect these notions,
but both sets of terminology seem too well-established to change, and
it should be clear below what meaning the word ``generic'' has in each
context.

A weaker way to obtain a notion of asymptotic computability is to
place no restrictions on how a description may fail to agree with the
function described, and require only that these failures be
restricted to a negligible set of inputs.

\begin{defn}
A \emph{dense description} of a (total) function $g : \omega \to \omega$ is
a partial function $f : \omega \to \omega$ such that
$f(n) = g(n)$ on a set of density $1$. A function is
\emph{densely computable} if it has a computable dense description.
\end{defn}

Notice that a dense description of $g$ is the same as a generic
description of some coarse description of $g$. Of course, any function
that is either generically or coarsely computable must be densely
computable, as all generic or coarse descriptions are dense
descriptions.

At the other extreme, we have descriptions that are required to either
answer correctly or halt with a signal that they will not give any
answer. We denote this signal by the symbol $\square$.

\begin{defn} \label{defnedc}
For a function $f : \omega \to \omega \cup \{\square\}$, the
\emph{strong domain} of $f$ is $f^{-1}(\omega)$.  Let $g : \omega \to
\omega$. A \emph{strong partial description} of $g$ is a (total)
function $f : \omega \to \omega \cup \{\square\}$ such that
$f(n)=g(n)$ on the strong domain of $f$.  An \emph{effective dense
description} of $g$ is a strong partial description of $g$ with
strong domain of density $1$. A function is \emph{effectively densely
computable} if it has a computable effective dense description.
\end{defn}

Note that a function is effectively densely computable if and only if
it has a partial computable generic description $f : \omega \to \omega$ such
that the domain of $f$ is computable. However, it is more convenient
for the purposes of relativization to work with total functions when
possible. Furthermore, as we will see, the totality of effective dense
descriptions causes them to behave more like coarse descriptions than
like generic ones in many ways.

Any function that is effectively densely computable is both
generically and coarsely computable, since we can modify an effective
dense description $f$ into either a generic description $f_0$ (by
having $f_0(n)\diverges$ whenever $f(n)=\square$) or a coarse
description $f_1$ (by defining $f_1(n)=0$ whenever $f(n)=\square$). We
can think of a computable effective dense description for a function
$g$ as an algorithm that, for each $n$, either correctly computes
$g(n)$ or announces that it cannot do so, such that the set on which
the algorithm fails to give an answer is negligible. Such algorithms
arise for instance in the case of problems that have computably
bounded generic-case complexity, in the sense of~\cite{KMSS}.

In addition to defining generic-case complexity, Kapovich, Myasnikov,
Schupp, and Shpilrain~\cite{KMSS} introduced the notion of generic
computability. Jockusch and Schupp~\cite{JS} then began to study it
from a computability-theoretic viewpoint. They also defined coarse
computability, although unbeknownst to them, that notion had already
been considered by Terwijn~\cite[Section 3.3]{T}. He was studying the
extent to which an incomplete set can resemble a complete one. As he
noted, versions of this question were also the focus of a much earlier
paper by Lynch~\cite{L} in which she answered a question of
Meyer~\cite{M}. In these two early works, the relevant notion was what
we call effective dense computability. Indeed, Meyer began by saying
that, ``The set of valid sentences of first-order predicate calculus
is not recursive, but the hope for general approaches to mechanical
theorem proving is that a reasonable fraction of the ``interesting''
sentences are effectively decidable. Ignoring the qualification that
sentences [be] ``interesting,'' we ask what fraction of the sentences
can be classified effectively as valid or invalid.'' He then gives the
following definition to formulate his question in ``purely
recursion-theoretic terms'': A set $C$ is \emph{approximable to within
  $\epsilon$} if there exist computable sets $A \subseteq C$ and $B
\subseteq \overline{C}$ such that $\limsup_n \frac{|[0,n-1] \setminus
  (A \cup B)|}{n} \leq \epsilon$. It is easy to see that an
effectively densely computable set is, in his terminology, one that is
approximable to within $0$. As far as we know, there has been no
further work on effective dense computability since Lynch's paper, and
dense computability has not been defined prior to this paper (although
it is mentioned by Cholak and Igusa~\cite{CI} in a
simultaneously-written paper).

As suggested by Meyer's definition, we can generalize each of our
notions to sets of lower density and define computability bounds as
follows. For generic computation, this was done by Downey, Jockusch,
and Schupp~\cite[Definition 6.9]{DJS}. They used the terms ``computable at density
$r$'' and ``asymptotic computability bound'', but since we are
considering several notions of asymptotic computation, we adopt the
terminology used by Hirschfeldt, Jockusch, McNicholl, and
Schupp~\cite[Section 1]{HJMS}, who also gave the corresponding definitions for
coarse computability.

\begin{defn}
We say that a function $g : \omega \to \omega$ is \emph{partially computable at density $r$} if there
is a computable partial description $f$ of $g$ such that $\underline{\rho}(\dom
f) \geq r$. The \emph{partial computability bound} of $g$ is
\[
\alpha(g) : = \sup\{r : \mbox{$g$ is partially computable at density
$r$}\}. 
\]

We say that $g$ is \emph{coarsely computable at density $r$} if there
is a (total) computable function $f$ such that $\underline{\rho}(\{n
: f(n)=g(n)\}) \geq r$. The \emph{coarse computability bound} of $g$
is
\[
\gamma(g) : = \sup\{r : \mbox{$g$ is coarsely computable at density
$r$}\}. 
\]
\end{defn}

Notice that a function is generically computable if and only if it is partially
computable at density $1$, and coarsely computable if and only if it is
coarsely computable at density $1$. It was pointed out
in~\cite[Observation 6.10]{DJS} that every nonzero c.e.\@ degree contains a
c.e.\@ set $A$ such that $\alpha(A)=1$ but $A$ is not generically computable,
and in~\cite[Theorem 3.3(ii)]{HJMS} that every nonzero c.e.\@ degree contains a
c.e.\@ set $B$ such that $\gamma(B)=1$ but $B$ is not coarsely computable. It
was also shown in~\cite[Lemma 1.7]{HJMS} that $\alpha(g) \leq \gamma(g)$ for
all $g$. (The proof was given for sets, but works in the setting of functions
as well.)

For our other notions of asymptotic computation, we have the following
analogous definitions.

\begin{defn}
\label{def:deltabeta}
Let $g : \omega \to \omega$.   We say that $g$ is \emph{weakly partially computable at density $r$} if there is a partial computable function $f$ such that
$\underline{\rho}(\{n : f(n)=g(n)\}) \geq r$. The \emph{weak partial
computability bound} of $g$ is
\[
\delta(g) : = \sup\{r : \mbox{$g$ is weakly partially computable at
density $r$}\}. 
\]

We say that $g$ is \emph{strongly partially computable at density $r$}
if there is a strong partial description $f$ of $g$ such that the
lower density of the strong domain of $f$ is at least $r$. The
\emph{strong partial computability bound} of $g$ is
\[
\beta(g) : = \sup\{r : \mbox{$g$ is strongly partially computable at
density $r$}\}. 
\]
\end{defn}

A function is densely computable if and only if it is weakly partially
computable at density $1$, and effectively densely computable if and only if it
is strongly partially computable at density $1$. As we will see in the
next section, the above computability bounds are not in fact new
concepts, because $\delta(g)=\gamma(g)$ and $\beta(g)=\alpha(g)$ for
all $g$. For a set $A$, being strongly partially computable at density
$r$ is the same as being approximable to within $1-r$, as defined by
Meyer~\cite{M}.

In Section~\ref{sec:asymptoticRelations} we discuss the relations
among our four notions of asymptotic computation and their
corresponding computability bounds, summarizing these in
Figure~\ref{fig:denseComputation} and Theorem \ref{boolean}. In Section~\ref{sec:relativization}
we relativize our notions, and define corresponding reducibilities and
degree structures. In Sections~\ref{sec:cofinite} and
\ref{sec:embeddings} we discuss natural embeddings of the Turing
degrees into these degree structures and the resulting notion of
quasiminimality, beginning by introducing some auxiliary notions
related to the mod-finite and cofinite reducibilities of Dzhafarov and
Igusa~\cite{DI}. In Section~\ref{sec:upperCones} we study the sizes of
upper cones and existence of minimal pairs in these
structures. Hirschfeldt, Jockusch, Kuyper, and Schupp~\cite[Theorem 5.2]{HJKS}
showed that nontrivial upper cones in the coarse degrees have measure
$0$, and used this fact to show the existence of minimal pairs in the
coarse degrees. We show that nontrivial upper cones in the generic,
dense, and effective dense degrees also have measure $0$, and show the
existence of minimal pairs in the dense degrees. Whether there are
minimal pairs in the generic or effective dense degrees remains
unknown. In Section~\ref{sec:questions}, we gather several open
questions.

\section{Comparing notions of asymptotic computability}
\label{sec:asymptoticRelations}

Our first order of business is to establish the relations between all
four notions of asymptotic computability, including their
corresponding computation bounds, which we summarize in
Figure~\ref{fig:denseComputation}.

\begin{figure}[ht]
\[
\xymatrix{
& *+[F]+\hbox{edc} \ar[dl] \ar[dr] & & \\
*+[F]+\hbox{cc} \ar[dr] & & *+[F]+\hbox{gc} \ar[dl] \ar[dr] & \\
& *+[F]+\hbox{dc} \ar[dr] & & *+[F]+\hbox{$\alpha=1$} \ar[dl] \\
& & *+[F]+\hbox{$\gamma=1$} &
}
\]
\caption{The graph of implications between notions of asymptotic
computability, including computability bounds. All implications
shown are strict, and all not shown are false. We have abbreviated
coarse computability by ``cc'', generic computability by ``gc'', and
(effective) dense computability by ``(e)dc''.}
\label{fig:denseComputation}
\end{figure}
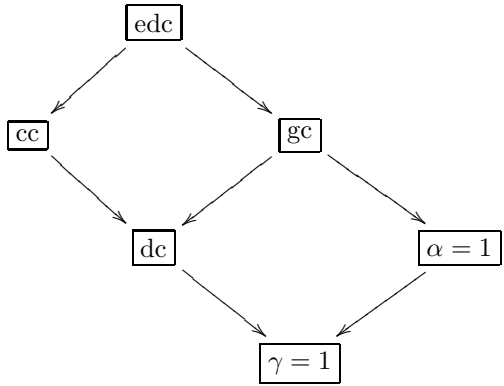

We have already noted that effective dense computability implies both
generic and coarse computability, which in turn both imply dense
computability. Jockusch and Schupp~\cite[Theorems 2.15 and 2.26]{JS} showed that there are
c.e.\@ sets that are coarsely computable but not generically computable, and
vice-versa, and hence that generic and coarse computability are
incomparable. It follows that effective dense computability is
strictly stronger than both generic and coarse computability, while
dense computability is strictly weaker than both.

Turning to the bottom part of Figure~\ref{fig:denseComputation}, we
begin by showing that we do not need to add the computability
bounds $\beta$ and $\delta$ from Definition~\ref{def:deltabeta} to it.

\begin{prop}
\label{prop:abgd}
Let $g:\omega\to\omega$. Then $\alpha(g)=\beta(g)$ and
$\gamma(g)=\delta(g)$.
\end{prop}

\begin{proof}
Suppose $g$ is partially computable at density $d$, as witnessed by
the partial description $f$, and let $\epsilon>0$. By Downey, Jockusch,
and Schupp~\cite[Theorem 3.9]{DJS}, since $\dom f$ is c.e., it has a
computable subset $B$ such that $\underline{\rho}(B) >
\underline{\rho}(\dom f)-\epsilon$. Let $h : \omega \to \omega \cup
\{\square\}$ be the total computable function defined by
\[
h(n)=
\begin{cases}
f(n) & \text{if } n \in B\\
\square & \text{if } n \notin B.
\end{cases}
\]
Then $h$ is a strong partial description of $g$ with strong domain of
lower density at least $d-\epsilon$. Since $\epsilon$ is arbitrary, we
have $\beta(g) \geq \alpha(g)$. But clearly $\alpha(g) \geq \beta(g)$,
so in fact $\alpha(g)=\beta(g)$.

Now suppose $g$ is weakly partially computable at density $d$, as
witnessed by the partial computable function $f$. Let $S=\{n :
f(n) = g(n)\}$ and let $\epsilon>0$. Again there is a computable
$B \subseteq \dom f$ such that $\underline{\rho}(B) >
\underline{\rho}(\dom f)-\epsilon$. Let $h$ be the total computable
function defined by
\[
h(n)=
\begin{cases}
f(n) & \text{if } n \in B\\
0 & \text{if } n \notin B.
\end{cases}
\]
Since $\underline{\rho}(S) \geq d$ and $S \subseteq \dom f$, we have
$\underline{\rho}(S \cap B) \geq d-\epsilon$. Since $S \cap B
\subseteq \{n : h(n)=g(n)\}$, it follows that $A$ is coarsely
computable at density $d-\epsilon$. Since $\epsilon$ is arbitrary, we
have $\gamma(g) \geq \delta(g)$. But clearly $\delta(g) \geq
\gamma(g)$, so in fact $\gamma(g)=\delta(g)$.
\end{proof}

This result also provides the last remaining implication in
Figure~\ref{fig:denseComputation}.

\begin{cor}
If $g$ is densely computable, then $\gamma(g)=1$.
\end{cor}

To complete the justification of Figure~\ref{fig:denseComputation}, it
remains only to separate $\alpha=1$ from coarse and dense
computability. As shown by Jockusch and Schupp~\cite[proof of Proposition 2.15]{JS}, the
usual construction of a simple set can be modified to yield a simple
set $A$ of density $0$. Then $A$ is coarsely computable but
$\alpha(A)=0$. Thus coarse computability does not imply $\alpha=1$. We
now separate $\alpha=1$ from dense computability.

\begin{prop} \label{ndc}
There is a c.e.\@ set $A$ such that $\alpha(A)=1$ but $A$ is not densely
computable.
\end{prop}

\begin{proof}
Let $A$ be defined as follows. For each $e,m \in \omega$ such that
$\Phi_e(2^ek)=0$ for at least half of the odd numbers $k<m$, enumerate
$2^ek$ into $A$ for all odd numbers $k<m$. Then, for each $e \in
\omega$, either $2^ek \in A$ for all odd numbers $k$, or $2^ek \in A$
for only finitely many odd numbers $k$. From this fact it follows
easily that $\alpha(A)=1$.

Now fix $e \in \omega$. If there are infinitely many $m$ such that
$\Phi_e(2^ek)=0$ for at least half of the odd numbers $k<m$, then for
each such $m$, we have that $\Phi_e(2^ek) = 0 \neq 1 = A(2^ek)$ for at
least half of the odd numbers $k<m$, and hence $\Phi_e$ cannot be a
dense description of $A$. Otherwise, for all sufficiently large $m$,
for more than half of the odd numbers $k<m$, we have that
$\Phi_e(2^ek)$ either diverges or converges to a number different from
$0$, while $A(2^ek)=0$. Again, $\Phi_e$ cannot be a dense description
of $A$.
\end{proof}

We thus have a full justification of the implications and
nonimplications in Figure~\ref{fig:denseComputation}, but we can
actually say more. Jockusch and Schupp~\cite[Theorem 2.22]{JS} built a
c.e.\@ set of density $1$ with no computable subset of density
$1$. Such a set is both generically and coarsely computable, but is
not effectively densely computable. Thus effective dense computability
is strictly stronger than the conjunction of generic and coarse
computability. Paul Schupp [personal communication] showed that there
is a c.e.\@ set $C$ such that $C$ is densely computable but neither
coarsely computable nor generically computable, which implies that
dense computability is strictly weaker than the disjunction of generic
and coarse computability.    He obtained such a $C$ as $A \oplus B$, where
$A$ is coarsely computable but not generically computable and $B$ is generically
computable but not coarsely computable.    This ``join" method will be crucial in the proof
of the more
general Theorem \ref{boolean} below.     Subsequently but independently, Justin
Miller [personal communication] showed that there is a set $C
\leq\sub{T} \emptyset'$ such that $C$ is densely computable but
neither coarsely computable nor generically computable and,
furthermore, $\alpha(C) = 1$.

In fact, we will show in Theorem \ref{boolean} below that any Boolean combination
of the properties in Figure~\ref{fig:denseComputation} is realizable by a c.e.\@ set,
provided it is not directly ruled out by the implications given in the
figure. The following lemma, whose proof follows easily from the
definitions, is basic to the proof.

\begin{lem}
\label{join}
Let $P$ be any of the six properties in
Figure~\ref{fig:denseComputation}. Then for any sets $A$ and $B$,  the
property $P$
holds of $A \oplus B$ if and only if $P$ holds of both $A$ and $B$.
\end{lem}

\begin{thm}   \label{boolean}
Let $V$ be the set of six vertices of the directed graph given in Figure~\ref{fig:denseComputation}.   For $p, q \in V$, define $q \leq p$ to mean that either $p = q$ or there is a directed path from $p$ to $q$ (so the property $p$ implies the property $q$).    Let $P$ and $Q$ be subsets of $V$.   Then the following statements are equivalent:

\begin{enumerate}[\rm (i)]
  
\item   There is a c.e.\@ set that satisfies every $p \in P$ and no $q
  \in Q$.

\item There is a set that satisfies every $p \in P$ and no $q
  \in Q$.

\item There is a function that satisfies every $p \in P$ and no $q
  \in Q$.
  
\item    There do not exist $p \in P$ and $q \in Q$ such that $q \leq
  p$.
  
\end{enumerate}
\end{thm}

\begin{proof}
Clearly, (i) implies (ii), which in turn implies (iii), and it is
immediate from the previous results of this section that (iii) implies
(iv).

Now assume that (iv) holds in order to prove (i).  Let $M$ be the set
of maximal elements
of $P$ under $\leq$.  Then $P$ and $M$ have the same downward closure
under $\leq$.  It follows that the meaning of (iv) does not change if
$P$ is replaced by $M$.  Also, the meaning of (i) does not change if
$P$ is replaced by $M$ since, by the results of this section, the set
of elements of $V$ true of a given set $C$ is closed downward and,
again, $P$ and $M$ have the same downward closure.  Thus, we assume
without loss of generality that $P = M$.  Since $M$ is an antichain
and, by inspection of Figure~\ref{fig:denseComputation} all antichains
have size at most $2$, we have $|P| \leq 2$.  The result is easy if
either $P$ or $Q$ is empty, so we may assume that $|P| \in \{1,2\}$.

\textbf{Case 1}.   $|P| = 1$.   Let $P = \{p\}$.    For each $q \in Q$,
we do not have $q \leq p$, so by the previous results in this section
there is a c.e.\@ set $A_q$ that satisfies $q$ but not $p$.   Let $C =
\bigoplus_{q \in Q} A_q$. Then $C$ witnesses (i), by Lemma
\ref{join}.

\textbf{Case 2}.  $|P| = 2$.    Since $P$ is an antichain, it follows
from inspection of Figure~\ref{fig:denseComputation} that $P$ is one
of the following three sets: $\{cc, gc\}$, $\{cc, \alpha = 1\}$, and
$\{dc , \alpha = 1\}$.

\quad\textbf{Case 2a}.   $P = \{cc, gc\}$.   Then, by (iii), $Q =
\{edc\}$, so (i) asserts that there is a c.e.\@ set $C$ that is coarsely
computable and generically computable but not effectively densely
computable.    The existence of such a set $C$ follows from the
discussion two paragraphs before Proposition 2.1, i.e., let $C$ be a
c.e.\@ set of density $1$ that has no computable subset of density $1$.

\quad\textbf{Case 2b}.   $P = \{cc, \alpha = 1\}$.    Then, by (iii),
$Q \subseteq  \{edc, gc\}$.        Let $B$ be a noncomputable c.e.\@
set and let $C = \{2^nk : n \in B\, \wedge\,
k \textrm{ odd}\}$.   Then $C$ is c.e., $C$ is coarsely computable by Theorem 2.19
of \cite{JS}, and $\alpha(C) = 1$ as in the proof of Proposition
\ref{ndc}.    Also, since $B$ is not computable, $C$ is not
generically computable by Observation 2.11 of \cite{JS}.    Hence, $C$
is also not effectively densely computable, and so witnesses that (i)
holds.

\quad\textbf{Case 2c}.   $P = \{dc, \alpha = 1\}$. Then $Q \subseteq
\{cc, gc, edc\}$ by (iii).
As shown in the previous case, there is a c.e.\@ set $A$ that is coarsely
computable but not generically computable, and has $\alpha(A)=1$. As
mentioned above, Jockusch and Schupp~\cite[Theorem 2.26]{JS} showed
that there is a c.e.\@ set
$B$ that is generically computable but not coarsely computable. Note
that both $A$ and $B$ are densely computable, and $\alpha(B)=1$. Let
$C = A \oplus B$. By Lemma~\ref{join}, $C$ is densely
computable but neither generically computable nor coarsely computable,
and $\alpha(C) = 1$. It follows that $C$ is also not effectively
densely computable. So $C$ witnesses that (i) holds.

Thus (i) holds in all cases and the proof is complete.
\end{proof}

\section{Reducibilities related to asymptotic computation}
\label{sec:relativization}

Notions of asymptotic computability can be relativized. For instance,
a function is coarsely computable relative to an oracle $X$ if it has
an $X$\-/computable coarse description. While this process does not
yield transitive relations, we can also define transitive
reducibilities based on our notions, capturing the idea of using
partial information about a function to compute partial information
about another. In each case, there are both a nonuniform and a uniform
version, and there do not seem to be good reasons to prefer one over
the other in general. For coarse and effective dense computability,
the definitions are straightforward. 

\begin{defn}
We say that $f$ is \emph{nonuniformly coarsely reducible} to $g$, and
write $f \leq\sub{nc} g$, if every coarse description of $g$ computes
a coarse description of $f$.

We say that $f$ is \emph{uniformly coarsely reducible} to $g$, and
write $f \leq\sub{uc} g$, if there is a Turing functional $\Phi$ such
that if $d$ is a coarse description of $g$, then $\Phi^d$ is a coarse
description of $f$.

We say that $f$ is \emph{nonuniformly effectively densely reducible}
to $g$, and write $f \leq\sub{ned} g$, if every effective dense
description of $g$ computes an effective dense description of
$f$.

We say that $f$ is \emph{uniformly effectively densely reducible} to
$g$, and write $f \leq\sub{ued} g$, if there is a Turing functional
$\Phi$ such that if $d$ is an effective dense description of $g$, then
$\Phi^d$ is an effective dense description of $f$.
\end{defn}

Generic and dense descriptions are partial functions, so we cannot use
them directly as oracles, but we can use their graphs. It would not
make sense to use Turing functionals in defining generic and dense
reducibility, however, because our reductions should not be able to
use information about where partial functions are not
defined. (For instance, from the graph of a dense description $f$ of a
set $A$, we can compute an effective dense description of $A$, since
we can determine for any given $n$ whether $f(n)\converges$ simply by
checking whether $(n,0) \in \gra(f)$ or $(n,1) \in \gra(f)$.) Thus we
use enumeration operators, which can use only positive facts. Recall
that an \emph{enumeration operator} is a c.e.\@ collection of pairs
$(F,k)$ with each $F$ finite. For an oracle $X$,  let $W^X = \{k : (\exists (F,k) \in W)\, F \subseteq X\}$.    By definition,  $Y$ is \emph {enumeration reducible to} $X$ if there is an
enumeration operator $W$ such that $Y = W^X$.    For partial functions $f$ and $g$, we say that
$f$ is \emph{enumeration reducible to} $g$ if $\gra(f)$ is enumeration reducible to $\gra(g)$.
Similarly, for partial functions $f$ and $g$, if $W^{\gra(g)} = \gra(f)$, we identify
$W^{\gra(g)}$ with $f$.

\begin{defn}
We say that $f$ is \emph{nonuniformly generically reducible} to $g$,
and write $f \leq\sub{ng} g$, if for every generic description $d$ of
$g$, there is a generic description $e$ of $f$ such that 
$e$ is enumeration reducible to $d$.

We say that $f$ is \emph{uniformly generically reducible} to $g$, and
write $f \leq\sub{ug} g$, if there is an enumeration operator $W$ such
that if $d$ is a generic description of $g$, then $W^{\gra(d)}$ is a
generic description of $f$.

We say that $f$ is \emph{nonuniformly densely reducible} to $g$, and
write $f \leq\sub{nd} g$, if for every dense description $d$ of $g$,
there is a dense description $e$ of $f$ such that 
$e$ is enumeration reducible to $d$.

We say that $f$ is \emph{uniformly densely reducible} to $g$, and
write $f \leq\sub{ud} g$, if there is an enumeration operator $W$ such
that if $d$ is a dense description of $g$, then $W^{\gra(d)}$ is a
dense description of $f$.
\end{defn}

Each of the reducibilities above induces an equivalence relation on
$\omega^\omega$ (or on $2^\omega$, if we wish to consider sets rather
than functions), from which a degree structure arises as usual. We
will be concerned with the degrees of sets, but it is worth noting
that, unlike in the case of the Turing degrees, it is not clear that
the degrees of functions coincide with those of sets. (See
Open Question~\ref{q:degrees}.)

We also note that while nonimplications between notions of asymptotic
computability immediately yield nonimplications between the
corresponding reducibilities, the same is not true of
implications. Thus several of the relationships between these
reducibilities remain open. (See Open Question~\ref{q:reds}.) In the
coarse and generic cases, Dzhafarov and Igusa have shown that uniform
reducibility is strictly stronger than nonuniform reducibility
(see~\cite[Theorems 2.6 and 2.7]{HJKS}). In
Section~\ref{sec:embeddings} we will see that this fact is also true
in the dense and effective dense cases.

A different and often more convenient approach to the definition of
generic reducibility, which has been used in several papers, is to
use Turing functionals but replace partial descriptions with
time-dependent representations known as partial oracles. The same can
be done for dense reducibility.

\begin{defn}
Let $g : \omega \to \omega$. A \emph{partial oracle} for $g$ is a set
$P$ such that if $\langle n,k,s \rangle \in P$ then $g(n)=k$. A
\emph{coarse partial oracle} for $g$ is a set that is a partial oracle
for some coarse description of $g$. The \emph{domain} of a partial
oracle $P$ is the set of $n$ such that $\langle n,k,s \rangle$ is in
$P$ for some $k$ and $s$.

A \emph{generic oracle} for $g$ is a partial oracle for $g$ with
domain of density $1$. A \emph{dense oracle} for $g$ is a coarse
partial oracle for $g$ with domain of density $1$.
\end{defn}

It is sometimes technically convenient to add to the definition of a
partial oracle $P$ the condition that for each $n$ and $k$ there is at
most one $s$ with $\langle n,k,s \rangle \in P$, as is done
in~\cite{CI}, but this modification does not change any of the notions
defined using partial oracles.

It is easy to see that if $h \leq\sub{ug} g$ then there is a Turing
functional $\Phi$ such that if $P$ is a generic oracle for $g$, then
$\Phi^P$ is a generic oracle for $h$, and that the analogous facts
hold for $\leq\sub{ng}$, $\leq\sub{ud}$, and $\leq\sub{nd}$.
Conversely, Igusa~\cite[Proposition 3.10]{I1} (see also Cholak and
Igusa~\cite[Observation 2.9]{CI}) showed that if such a
$\Phi$ exists then
$h \leq\sub{ug} g$. (He stated this fact for generic descriptions of
sets rather than functions, but the argument works just as well in the
more general setting of functions.) We now show that the analogous
fact in the nonuniform case, which was left open in the above papers, also
holds.

\begin{prop}
\label{prop:gendef}
For all functions $g$ and $h$, we have $h \leq\sub{ng} g$ if and only if every
generic oracle for $g$ computes a generic oracle for $h$.
\end{prop}

\begin{proof}
As mentioned above, the ``only if'' direction is clear. Suppose that
every generic oracle for $g$ computes a generic oracle for $h$ and fix
a generic description $f$ of $g$.
  
Say that $\sigma \in 2^{<\omega}$ is \emph{acceptable} if $\sigma(\langle
n,i,s \rangle)=1 \Rightarrow f(n) = i$. We force over the
acceptable strings ordered by extension. Let $G$ be sufficiently
generic for this notion of forcing. If $\sigma \prec G$ then $\sigma$
is acceptable, so if $\langle n,i,s \rangle \in G$ then
$f(n) = i$, and hence $g(n)=i$. Thus $G$ is a partial oracle
for $g$. If $f(n)= i$ then the set of acceptable $\sigma$
such that $\sigma(\langle n,i,s \rangle)=1$ for some $s$ is dense, and
hence $\langle n,i,s \rangle \in G$ for some $s$. Thus $\dom G \supseteq \dom f$,
and hence $\rho(\dom G)=1$.

Since $G$ is a generic oracle for $g$, there is a Turing functional $\Phi$
such that $\Phi^G$ is a generic oracle for $h$. Consider the set of
acceptable $\sigma$ such that there exist $m$, $j$, and $t$ with $\Phi^\sigma(\langle m,j,t
\rangle) = 1$ and $h(m) \neq j$. No initial segment of $G$
can be in this set, so $G$ must avoid this set. That is, there is a
$\tau \prec G$ such that for all acceptable $\sigma \succ \tau$, if
$\Phi^\sigma(\langle m,j,t \rangle)= 1$ then $h(m)=j$.

Now define an enumeration operator $W$ as follows. For each finite $F \subset \omega \times
\omega$, if there is a $\sigma \succ \tau$ (with $\sigma$ not
necessarily acceptable) such that $\sigma(\langle n,i,s \rangle)=1
\Rightarrow (n,i) \in F$, and $\Phi^\sigma(\langle m,j,t
\rangle) = 1$, then add the pair $(F,(m,j))$ to $W$.

Let $d=W^{\gra(f)}$.  If $F \subset \gra(f)$ is finite, then any
$\sigma$ as in the previous paragraph is acceptable, so by the choice
of $\tau$, if $(m,j) \in d$ then $h(m)=j$. In particular, $d$ is single-valued, and
we identify it with the partial function of which it is the graph.   Thus to show that $d$ is a
generic description of $h$, it suffices to show that $\rho(\dom
d)=1$. If $m \in \dom \Phi^G$ then there is a $\sigma$ with $\tau
\prec \sigma \prec G$ such that $\Phi^\sigma(\langle m,j,t
\rangle) = 1$ for some $j,t$. Since $\sigma$ is acceptable,
there is a finite $F \subset \gra(f)$ such that $(F,(m,j)) \in W$, so
$m \in \dom d$. Thus $\dom d \supseteq \dom \Phi^G$, and hence
$\rho(\dom d)=1$.
\end{proof}

Thus in both the uniform and nonuniform cases, the two ways to define
generic reducibility agree. The above proof can easily be adapted to
certain other reducibilities, such as the infinite-information
reducibility introduced by Dzhafarov and
Igusa~\cite{DI}. Unfortunately, this is not the case for dense
reducibility, as there seems to be no good analog to the existence of
$\tau$ in that proof. Indeed, the aforementioned proofs in~\cite{I1}
and~\cite{CI}
for the uniform case also seem to fail for dense reducibility. Thus
we do not know whether defining dense reducibility using dense oracles
and Turing functionals would yield the same notions. (See Open
Question~\ref{q:dense}.) Fortunately, the
results in this paper do not depend on which version of dense
reducibility we choose.

\section{Cofinite reducibilities}
\label{sec:cofinite}

As shown by  Jockusch and Schupp \cite{JS}, Dzhafarov and
Igusa~\cite{DI}, and Hirschfeldt, Jockusch,
Kuyper, and Schupp~\cite{HJKS}, there are natural embeddings of the
Turing degrees into the coarse and generic degrees. As we will see in
the next section, the same process works for the dense and effective
dense degrees as well, by arguments similar to the ones in these
papers. We will follow the version in~\cite{DI}, which filters through
the notions of cofinite and mod-finite reducibility. In our case, we
need to introduce two variants on those notions.

\begin{defn}
A \emph{cofinite description} of a function $g$ is a partial
description of $g$ with cofinite domain. A function $h$ is
\emph{cofinitely reducible} to a function $g$, written as $h
\leq\sub{cf} g$, if there is an enumeration operator $W$ such that if
$f$ is a cofinite description of $g$ then $W^{\gra(f)}$ is the graph of
a cofinite description of $h$.

A \emph{mod-finite description} of a function $g$ is a function $f$
such that $f(n)=g(n)$ for almost all $n$. A function $h$ is
\emph{mod-finitely reducible} to a function $g$, written as $h
\leq\sub{mf} g$, if there is a Turing functional $\Phi$ such that if
$f$ is a mod-finite description of $g$ then $\Phi^f$ is a mod-finite
description of $h$.

A \emph{weak cofinite description} of a function $g$ is a cofinite
description of a mod-finite description of $g$, i.e., a partial
function $f$ such that $f(n) = g(n)$ for almost all $n$. A
function $h$ is \emph{weak-cofinitely reducible} to a function $g$,
written as $h \leq\sub{wcf} g$, if there is an enumeration operator
$W$ such that if $f$ is a weak cofinite description of $g$ then
$W^{\gra(f)}$ is the graph of a weak cofinite description of $h$.

A \emph{strong cofinite description} of a function $g$ is a strong
partial description of $g$ with cofinite strong domain. A function $h$
is \emph{strong-cofinitely reducible} to a function $f$, written as
$h \leq \sub{scf} g$, if there is a Turing functional $\Phi$ such that
if $f$ is a strong cofinite description of $g$ then $\Phi^f$ is a
strong cofinite description of $g$.
\end{defn}

In~\cite{DI}, the definitions of mod-finite and cofinite reducibility
also required that $\Phi^g=h$, but as shown in that paper, dropping
this condition does not affect either definition.  Notice that the
nonuniform version of each of these reducibilities is just Turing
reducibility. These reducibilities induce degree structures in the
usual way.

A \emph{cofinite oracle} for a function $g$ is a partial oracle for
$g$ with cofinite domain. The original definition of cofinite
reducibility in~\cite{DI} is in terms of cofinite oracles and Turing
functionals, but essentially the same argument that shows that uniform
generic reducibility is equivalent to the version defined using
generic oracles shows that $h \leq\sub{cf} g$ if and only if there is a Turing
functional $\Phi$ such that if $P$ is a cofinite oracle for $g$ then
$\Phi^P$ is a cofinite oracle for $h$. We do not know whether the
analogous fact holds  for weak cofinite reducibility, however. (See
Open Question~\ref{q:dense}.)

Our main interest in wcf\-/reducibility and scf\-/reducibility here is
as tools in the study of dense and effective dense reducibility, but
it is of independent interest to consider how they fit into the
picture of measures of robust information coding in~\cite{DI}. To begin
to answer this question, we define another notion introduced
in~\cite{DI}.

\begin{defn}
We say that $h$ is \emph{use-bounded-from-below reducible} to $g$, and
write $h \leq\sub{ubfb} g$, if $h$ is computable from $g$ via a
functional $\Phi$ such that the smallest value of $g$ queried in
the computation of $\Phi^g(n)$ goes to infinity with $n$.
\end{defn}

Dzhafarov and Igusa~\cite[Theorem 3.1]{DI} showed that, in Cantor space,
mf\-/reducibility strictly implies ubfb\-/reducibility, which in turn
strictly implies cf\-/reducibility. Their proofs work in Baire space
as well, except for the implication between mf\-/reducibility and
ubfb\-/reducibility. Indeed, we have the following fact.

\begin{prop}
There are a function $g$ and a set $B$ such that $B \leq\sub{mf} g$
but $B \nleq\sub{cf} g$, and hence $B \nleq\sub{ubfb} g$.
\end{prop}

\begin{proof}
We will define a function $g$ and let $B=\{\langle n,k \rangle :
g(n)=k\}$. Then it is easy to see that $B \leq\sub{mf} g$.

We define $g$ by finite extensions as follows. Suppose we have defined
$g$ on $[0,n]$ and want to ensure that $B \nleq\sub{cf} g$ via a given
enumeration operator $W$.  Let $\mathcal{C}$ be the class of all
partial functions with finite domain that agree with $g$ on $[0, n]$
and are undefined at $n+1$.

\textbf{Case 1.} There are a $k
\in \omega$, an $i<2$, and an $f \in \mathcal C$ such that $(\langle
n+1,k \rangle, i) \in W^{\gra(f)}$.    Let $g(n) = f(n)$ on the domain of $f$.
Also, if $i=0$
then let $g(n+1)=k$, and otherwise let $g(n+1)=k+1$. This definition is
consistent since $n + 1$ is not in the domain of $f$, and it
ensures that $B(\langle n+1,k \rangle) \neq i$.  Finally, extend the definition
of $g$ so that the current domain of $g$ is a finite initial segment of the
natural numbers.      Then $W^{\gra(g)}$ contains $(\langle
n+1,k \rangle, i)$ and so is not of the form $\gra(h)$ for any partial description
$h$ of $B$.

\textbf{Case 2.} Otherwise. Then we simply set $g(n+1) = 0$. If $g_0$
is the partial function obtained from (the final version of) $g$ by
omitting $n+1$ from its domain, then $g_0$ is a cofinite description
of $g$ such that $W^{\gra(g_0)}$ contains no number of the form
$(\langle n+1,k \rangle, i)$ for $i<2$. Thus if $W^{\gra(g_0)}$ is the graph of
a partial function $h$, then $h$ cannot be defined as either $0$ or
$1$ on any argument of
the form $\langle n+1,k \rangle$, and hence $h$ cannot be a cofinite
description of any set.
\end{proof}

We also have the following implications.

\begin{prop}
If $h \leq\sub{scf} g$ then $h \leq\sub{ubfb} g$.
\end{prop}

\begin{proof}
Suppose that $h \leq\sub{scf} g$ via $\Phi$. By finitely modifying
$\Phi$ if necessary, we may assume that $\Phi^g=h$. Let $g_n$ be
defined by letting $g_n(m)=\square$ for $m<n$ and $g_n(m)=g(m)$ for $m
\geq n$. Note that $\Phi^{g_n}(k)\converges \in \{h(k),\square\}$ for
all $n$ and $k$. We can compute $h$ from $g$ as follows. On input $k$,
find the largest $n \leq k$ such that $\Phi^{g_n}(k) \neq \square$
(which must exist, since $\Phi^{g_0}=\Phi^g=h)$. Then
$h(k)=\Phi^{g_n}(k)$. This computation can be performed without
querying $g(m)$ for any $m<n$. For each $n$, we have that $\Phi^{g_n}$
is a strong cofinite description of $h$, so $\Phi^{g_n}(k) \neq
\square$ for almost all $k$. Thus the above procedure has
use bounded from below.
\end{proof}
  
\begin{prop}
If $h \leq\sub{ubfb} g$ then $h \leq\sub{wcf} g$.
\end{prop}

\begin{proof}
Suppose that $h \leq\sub{ubfb} g$ via $\Phi$. Define an enumeration
operator $W$ as follows. For each finite $F \subset \omega \times
\omega$ and each $n$, run the computation of $\Phi$ on input $n$, and
each time this computation queries its oracle on an input $m$, check
whether $(m,j) \in F$ for some $j$. If not, or if this is the case for
more than one $j$, then do nothing. Otherwise, provide the answer $j$
to $\Phi$'s query, and continue the computation. If this computation
ever halts with an output $i$, then enumerate the pair $(F,(n,i))$
into $W$.

Let $f$ be a weak cofinite description of $g$. Then $f$ is a cofinite
description of some mod-finite description $d$ of $g$. If $(F,(n,i))
\in W$ for some $F \subset \gra(f)$ then $\Phi^d(n)=i$, so
$W^{\gra(f)}$ is the graph of a partial function. Since $\Phi$ has its
use bounded from below on oracle $g$, if $n$ is sufficiently large
then there is a finite $F \subset \gra(f)$ such that all answers
provided to $\Phi$ during the simulation described above are the same
as $\Phi$ would get from $g$, and hence $(n,h(n)) \in
W^{\gra(f)}$. Thus $W^{\gra(f)}$ is a weak cofinite description of
$h$.
\end{proof}

\begin{cor}
\label{cor:cofinite}
For functions $g$ and $h$, if $h \leq\sub{scf} g$, then $h
\leq\sub{cf} g$ and $h \leq\sub{wcf} g$. For sets $A$ and $B$, if $B
\leq\sub{mf} A$, then $B \leq\sub{cf} A$ and $B \leq\sub{wcf} A$.
\end{cor}

Notice also that both wcf\-/reducibility and scf\-/reducibility are
implied by $1$\-/reducibility (for sets) and imply Turing
reducibility.  We do not know whether either of the implications in
the above propositions is strict, or what the relationships between
mf\-/reducibility and scf\-/reducibility, and between
cf\-/reducibility and wcf\-/reducibility are. (See Open
Question~\ref{q:cf}.)

\section{Embeddings of the Turing degrees and quasiminimality}
\label{sec:embeddings}

We can now define embeddings of the Turing degrees into each of our
asymptotic reducibilities. Since every function is Turing equivalent
to a set, we can work in Cantor space without loss of generality. Let
\[
\mathcal R(A)=\{2^nk : n \in A \, \wedge \, k \textrm{ odd}\}.
\]
Taking $J_n=[2^n,2^{n+1})$, define
\[
\widetilde{\mathcal R}(A) = \bigcup_{n\in A} J_n.
\]
Finally, let $\mathcal E(A) = \widetilde{\mathcal R}(\mathcal R(A))$.

Dzhafarov and Igusa~\cite[Proposition 3.3]{DI} showed that $\mathcal
R$ induces
embeddings of the Turing degrees into both the mod-finite and the
cofinite degrees, and that $\widetilde{\mathcal R}$ induces
embeddings of the cofinite degrees into the uniform generic degrees
and of the mod-finite degrees into the uniform coarse degrees. A
similar but easier argument shows that the latter map also induces
embeddings of the Turing degrees into both the nonuniform generic and
the nonuniform coarse degrees, since $\widetilde{\mathcal R}(A)
\equiv\sub{T} A$, and any coarse description of $\widetilde{\mathcal
R}(A)$ or dense oracle for $\widetilde{\mathcal R}(A)$ computes
$A$. Since $\mathcal R(A) \equiv\sub{T} A$, it follows that $\mathcal
E$ also induces these embeddings.

Thus we see that $\mathcal E$ induces embeddings of the Turing degrees
into the uniform and nonuniform generic and coarse
degrees. (See~\cite[Proposition 2.2] {HJKS} for a proof that these
embeddings are the
unique ones with certain natural properties.) We now
show that this is also the case for the uniform and nonuniform dense
and effective dense degrees, by adapting the arguments in~\cite{DI}.

\begin{lem} \label{descr}
\mbox{}
\begin{enumerate}[\rm 1.]

\item  From the graph of a weak cofinite description of $A$ we can
uniformly enumerate the graph of a dense description of
$\widetilde{\mathcal R}(A)$, and vice-versa.

\item From a strong cofinite description of $A$ we can uniformly
compute a strong dense description of $\widetilde{\mathcal R}(A)$, and
vice-versa.

\end{enumerate}
\end{lem}

\begin{proof}
Part 1: Suppose $f$ is a weak cofinite description of $A$ and define
$g$ by letting $g(m)=f(n)$ for $m \in J_n$. The graph of $g$ can be
uniformly enumerated from the graph of $f$, and if $f(n)=A(n)$ then
$g(m)=\widetilde{\mathcal R}(A)(m)$ for all $m \in J_n$, so $g$ is a
weak cofinite description, and hence a dense description, of
$\widetilde{\mathcal R}(A)$.

Now suppose $g$ is a dense description of $\widetilde{\mathcal R}(A)$
and define $f$ as follows. For each $n$, if we see $g(m)= i$
for more than half of the elements of $J_n$, then let
$f(n)= i$. The graph of $f$ can be uniformly enumerated from
the graph of $g$. If $n$ is large enough, then
$\rho_{2^{n+1}}(\{m : g(m) = \widetilde{\mathcal R}(A)(m)\}) >
\frac{3}{4}$. Since $|J_n|=2^n$, it follows that $g(m) = A(n)$
for more than half of the elements of $J_n$, and hence
$f(n) = A(n)$. Thus $f$ is a weak cofinite description of $A$.

Part 2: Suppose $f$ is a strong cofinite description of $A$ and define
$g$ by letting $g(m)=f(n)$ for $m \in J_n$. Then $g$ can be
uniformly computed from $f$, and if $f(n)=A(n)$ then
$g(m)=\widetilde{\mathcal R}(A)(m)$ for all $m \in J_n$, so $g$ is a
strong cofinite description, and hence a strong dense description, of
$\widetilde{\mathcal R}(A)$.

Now suppose $g$ is a strong dense description of $\widetilde{\mathcal
R}(A)$ and define the $g$-computable function $f$ as follows. For
each $n$, if $g(m)=i \neq \square$ for some $m \in J_n$ then let
$f(n)=i$, and otherwise let $f(n)=\square$. (Notice that this
definition makes sense, because if $m,m' \in J_n$ and both $g(m) \neq
\square$ and $g(m') \neq \square$, then $g(m)=A(n)=g(m')$.) If $f(n)
\neq \square$ then $f(n)=\widetilde{\mathcal R}(A)(m)$ for some $m \in
J_n$, and hence $f(n)=A(n)$. If $n$ is large enough, then
$\rho_{2^{n+1}}(\{m : g(m)=\widetilde{\mathcal R}(A)(m)\}) >
\frac{1}{2}$. Since $|J_n|=2^n$, it follows that $g(m)=A(n)$ for some
$m \in J_n$, and hence $f(n)=A(n)$. Thus $f$ is a strong cofinite
description of $A$.
\end{proof}

\begin{prop}
\mbox{}
\begin{enumerate}[\rm 1.]

\item The map $\widetilde{\mathcal R}$ induces an embedding of the
Turing degrees into the nonuniform dense degrees.

\item The map $\widetilde{\mathcal R}$ induces an embedding of the
Turing degrees into the nonuniform effective dense degrees.

\item The map $\widetilde{\mathcal R}$ induces an embedding of the
weak cofinite degrees into the uniform dense degrees.

\item The map $\widetilde{\mathcal R}$ induces an embedding of the
strong cofinite degrees into the uniform effective dense degrees.

\item The map $\mathcal R$ induces an embedding of the Turing degrees
into the weak cofinite degrees, and hence $\mathcal E$ induces an
embedding of the Turing degrees into the uniform dense degrees.

\item The map $\mathcal R$ induces an embedding of the Turing degrees
into the strong cofinite degrees, and hence $\mathcal E$ induces an
embedding of the Turing degrees into the uniform effective dense
degrees.
  
\end{enumerate}
\end{prop}

\begin{proof}
Part 1: Suppose that $B \leq\sub{T} A$. From a dense description of
$\widetilde{\mathcal R}(A)$, we can (nonuniformly) recover $A$, and
hence $B$, and hence $\widetilde{\mathcal R}(B)$. Thus
$\widetilde{\mathcal R}(B) \leq\sub{nd} \widetilde{\mathcal
R}(A)$. Conversely, suppose that $\widetilde{\mathcal R}(B)
\leq\sub{nd} \widetilde{\mathcal R}(A)$. From $A$ we can compute
$\widetilde{\mathcal R}(A)$, and hence obtain a dense description of
$\widetilde{\mathcal R}(B)$, from which we can (nonuniformly) recover
$B$. Thus $B \leq\sub{T} A$.

Part 2 follows by essentially the same argument.

Part 3: Suppose that $B \leq\sub{wcf} A$. By the first part of the
lemma, given a dense description of $\widetilde{\mathcal R}(A)$, we
can uniformly recover a weak cofinite description of $A$, from which
we can uniformly obtain a weak cofinite description of $B$, and hence,
again by the first part of the lemma, uniformly obtain a dense
description of $\widetilde{\mathcal R}(B)$. The other direction is
similar.

Part 4 follows by essentially the same argument, using the second part
of the lemma.

Part 5: If $\mathcal R(B) \leq\sub{wcf} \mathcal R(A)$ then $B
\equiv\sub{T} \mathcal R(B) \leq\sub{T} \mathcal R(A) \equiv\sub{T} A$. 
Conversely, suppose that $\Phi^A=B$. Define the
enumeration operator $W$ as follows. For each finite $F \subset \omega
\times \omega$ and each $n$ and $k$, if there is a $\sigma \in
2^{< \omega}$ such that $\Phi^\sigma(n)\converges$ and for each
$m<|\sigma|$, we have $(2^m(2(n+k)+1),\sigma(m)) \in F$, then
enumerate the pair $(F,(2^n(2k+1),\Phi^\sigma(n)))$ into $W$. Let $f$
be a weak cofinite description of $\mathcal R(A)$. Then $W^{\gra(f)}$
is the graph of a partial function $g$. If $n+k$ is sufficiently
large, then $g(2^m(2(n+k)+1)) = A(m)$ for all $m$, and hence
$g(2^n(2k+1)) = B(n)$. Hence $g$ is a weak cofinite
description of $\mathcal R(B)$. Thus $\mathcal R(B) \leq\sub{wcf}
\mathcal R(A)$.

Part 6 follows by a similar argument to part 1, using the fact that
from a strong cofinite description of $\mathcal R(A)$, we can
uniformly recover $A$.
\end{proof}

These embeddings can be used to separate the uniform and nonuniform
versions of dense and effective dense reducibility. In the latter
case, the proof is the same as the one for generic and coarse
reducibility, due to Igusa (see~\cite[Theorems 2.6 and 2.7]{HJKS}) and
based on results of
Dzhafarov and Igusa~\cite{DI}. Recall that a set $X$ is
\emph{autoreducible} if there is a Turing functional $\Phi$ such that
$\Phi^{X \setminus \{n\}}(n)=X(n)$ for all $n$. Examples of sets that
are not autoreducible include all $1$\-/random sets (as shown by Figueira,
Miller, and Nies~\cite[Proposition 8]{FMN}) and all $1$\-/generic sets
(see~\cite[Proposition 2.5]{HJKS}).

\begin{cor}
There are sets $A$ and $B$ such that $A \leq\sub{ned} B$ but $A
\nleq\sub{ued} B$.
\end{cor}

\begin{proof}
Let $X$ be a set that is not autoreducible, let $A=\mathcal E(X)$, and
let $B=\widetilde{\mathcal R}(X) $. By part 1 of Lemma~\ref{descr},
any effective dense description of
$B$ computes $X$, and hence computes $A$, so $A \leq\sub{ned}
B$. Dzhafarov and Igusa~\cite[Proof of Proposition 4.6]{DI} showed
that $\mathcal R(X) \nleq
\sub{cf} X$, so by Corollary~\ref{cor:cofinite}, $\mathcal R(X) \nleq
\sub{scf} X$. Since $\widetilde{\mathcal R}$ induces an embedding of
the scf-degrees into the ued-degrees, $A \nleq\sub{ued} B$.
\end{proof}

For dense reducibility, the analogous argument would filter through
weak cofinite reducibility, and it is not clear that assuming that $X$
is not autoreducible suffices to ensure that $\mathcal R(X) \nleq
\sub{wcf} X$. However, we can replace this assumption with a stronger
one mentioned by Hirschfeldt, Jockusch, Kuyper, and
Schupp~\cite[Section 2]{HJKS}.

\begin{defn}
A set $X$ is \emph{jump-autoreducible} if there is a Turing
functional $\Phi$ such that $\Phi^{(X \setminus \{n\})'}(n) = X(n)$
for all $n$.
\end{defn}

Not all sets are jump-autoreducible. Indeed, it was shown in an early
version of~\cite{HJKS} that $2$-generic sets and $2$-random sets are not
jump-autoreducible. These proofs are not in the final version of that
paper, as considering failures of jump-autoreducibility turned out not
to be necessary in studying coarse and generic reducibility. Since
we will use this concept in connection with dense reducibility, we
include these proofs here.

\begin{prop}[Hirschfeldt, Jockusch, Kuyper, and Schupp]
If $X$ is $2$-generic, then $X$ is not jump-autoreducible.
\end{prop}

\begin{proof}
Since $X$ is $2$-generic, $X$ is $1$-generic relative to $\emptyset'$.
Hence, by the relativized version of the proof in~\cite[Proposition 2.5]{HJKS} that
$1$-generic sets are not autoreducible, $X$ is not autoreducible relative
to $\emptyset'$.  However, the class of $1$-generic sets is uniformly
$\mathrm{GL}_1$, i.e., there exists a single Turing functional $\Psi$
such that for every $1$-generic $X$ we have $\Psi^{X \oplus
\emptyset'} = X'$, as can be verified by looking at the usual proof
that every $1$-generic is $\mathrm{GL}_1$ (see~\cite[Lemma
2.6]{J}). Of course, if $X$ is $1$-generic, then $X \setminus \{n\}$
is also $1$-generic for every $n$.  Thus from an oracle for $(X
\setminus \{n\}) \oplus \emptyset'$ we can uniformly compute $(X
\setminus \{n\})'$. Now, if $X$ is jump-autoreducible, from $(X
\setminus \{n\})'$ we can uniformly compute $X(n)$.  Composing these
reductions shows that $X(n)$ is uniformly computable from $(X
\setminus \{n\}) \oplus \emptyset'$, which contradicts our previous
remark that $X$ is not autoreducible relative to $\emptyset'$.
\end{proof}

In the proof above we used the fact that the $2$-generic sets are
uniformly $\mathrm{GL}_1$. For $2$-random sets this fact is almost
true, as expressed by the following lemma. The proof is adapted from
Monin~\cite[Proposition 4.2.5]{Mo}, where a generalization for higher
levels of randomness
is proved. Let $\mathcal U_0,\mathcal U_1,\ldots$ be a fixed universal
Martin-L\"of test relative to $\emptyset'$. The \emph{$2$-randomness
deficiency} of a $2$-random $X$ is the least $c$ such that $X \notin
\mathcal U_c$

\begin{lem}[after Monin~\cite{Mo}]
There is a Turing functional $\Theta$ such that, for any $2$-random
$X$ and any upper bound $b$ on the $2$-randomness deficiency of $X$,
we have $\Theta^{X \oplus \emptyset',b} = X'$.
\end{lem}

\begin{proof}
Let $\mathcal V_e=\{Z : e \in Z'\}$. The $\mathcal V_e$ are uniformly
$\Sigma^0_1$ classes, so we can define a function $f \leq\sub{T}
\emptyset'$ such that $\mu(\mathcal V_e \setminus \mathcal
V_e[f(e,i)])<2^{-i}$ for all $e$ and $i$, where $\mu$ is Lebesgue
measure on $2^\omega$. Then each sequence $\mathcal V_e \setminus
\mathcal V_e[f(e,0)],\mathcal V_e \setminus \mathcal
V_e[f(e,1)],\ldots$ is an $\emptyset'$-Martin L\"of test, and from $b$
we can compute a number $m$ such that if $X$ is $2$-random and $b$
bounds the $2$-randomness deficiency of $X$, then $X \notin \mathcal
V_e \setminus \mathcal V_e[f(e,m)]$. Then $X \in \mathcal V_e$ if and only if
$X \in \mathcal V_e[f(e,m)]$, which we can verify
$(X\oplus\emptyset')$-computably.
\end{proof}

\begin{prop}[Hirschfeldt, Jockusch, Kuyper, and Schupp]
If $X$ is $2$-random, then $X$ is not jump-autoreducible.
\end{prop}

\begin{proof}
Because $X$ is $2$-random, it is not autoreducible relative to
$\emptyset'$, as can be seen by relativizing the proof of Figueira,
Miller, and Nies~\cite[Proposition 8]{FMN} that no $1$-random set is
autoreducible.
To obtain a contradiction, assume that $X$ is jump-autoreducible via
some functional $\Phi$. It can be directly verified that there is a
computable function $f$ such that $f(n)$ bounds the randomness
deficiency of $X \setminus \{n\}$. Now let $\Psi^{Y \oplus
\emptyset'}(n) = \Phi^{\Theta^{Y \oplus \emptyset',f(n)}}(n)$. Then
$X$ is autoreducible relative to $\emptyset'$ via $\Psi$, which is a
contradiction.
\end{proof}

\begin{cor}
There are sets $A$ and $B$ such that $A \leq\sub{nd} B$ but $A
\nleq\sub{ud} B$.
\end{cor}

\begin{proof}
Let $X$ be a set that is not jump-autoreducible, let $A=\mathcal
E(X)$, and let $B=\widetilde{\mathcal R}(X) $. Any dense description
of $B$ computes $X$, and hence computes $A$, so $A \leq\sub{nd} B$. It
is now enough to show that $\mathcal R(X) \nleq\sub{wcf} X$. Since
$\widetilde{\mathcal R}$ induces an embedding of the wcf-degrees into
the ud-degrees, it will then follow that $A \nleq\sub{ud} B$.

Assume for a contradiction that there is an enumeration operator $W$
such that $W^f$ is a weak cofinite description of $\mathcal R(X)$ for
each weak cofinite description $f$ of $X$. Given an oracle $Y$ and an
$n$, think of $Y$ as a function and enumerate $W^{\gra(Y)}$. Let
$g(n,s)=0$ if at least half of the pairs $(2^n(2k+1),i)$ enumerated by
stage $s$ have $i=0$, and let $g(n,s)=1$ otherwise. Let $g(n) = \lim_s
g(n,s)$, or $g(n)\diverges$ if this limit does not exist. We can
uniformly compute the partial function $g$ from $Y'$. If $Y=X
\setminus \{n\}$ then $Y$ is a weak cofinite description of $X$, so
$W^X$ is a weak cofinite description of $\mathcal R(X)$, and hence
$g(n) = X(n)$. Thus we have a procedure for uniformly
computing $X(n)$ from $(X \setminus \{n\})'$, contrary to hypothesis.
\end{proof}

Another issue that has attracted some attention recently is that of
quasiminimality. For any of our asymptotic reducibilities r, an
r\-/degree is \emph{quasiminimal} if it is not zero and is not above any nonzero
degree in the image of the embedding induced by $\mathcal E$. We
expect the r\-/degrees of sufficiently random sets to be quasiminimal,
because if $X$ is not quasiminimal, then there is some noncomputable
set that is computable from all partial versions of $X$ (in whatever
sense of partial version is appropriate to r), and it should not be
possible to code a nontrivial amount of information into a random set
in a sufficiently redundant way for such information recovery from
partial versions to be possible. Indeed, Cholak and
Igusa~\cite[Proposition 5.5]{CI}
showed that every $1$\-/random set has quasiminimal degree in both the
uniform coarse and uniform generic degrees. Hirschfeldt, Jockusch,
Kuyper, and Schupp~\cite[remarks after Corollary 3.3]{HJKS} showed
that while this is not quite the
case in general for the nonuniform coarse and generic degrees, every
weakly $2$\-/random set does have quasiminimal degree in the
nonuniform coarse degrees. Cholak, Hirschfeldt, and Igusa
(see~\cite[Theorem 6.3]{CI}) showed that this is also the case in the nonuniform
generic degrees.

Hirschfeldt, Jockusch, Kuyper, and Schupp~\cite[Theorem 4.2]{HJKS}
also showed that
every $1$\-/generic set has quasiminimal degree in both the uniform
and nonuniform coarse degrees. Cholak and Igusa~\cite[Proposition
5.4]{CI} showed
that this is also the case in the uniform generic degrees, and Cholak,
Hirschfeldt, and Igusa (see~\cite[Proposition 6.9]{CI}) showed the
same for the nonuniform generic degrees.

A special case of Corollary~3.14 in~\cite{HJKS} is that if $X
\leq\sub{T} \emptyset'$ is $1$\-/random, then there is a promptly
simple set $A$ that is computable from every dense oracle for $X$. A
similar proof to that of Proposition~\ref{prop:gendef} then shows that
for each dense description $f$ of $X$, we have that $\mathcal E(A)$ is
enumeration
reducible to $f$, which implies that
$\mathcal E(A) \leq\sub{nd} X$ and $\mathcal E(A) \leq\sub{ned}
X$. Thus not every $1$\-/random set has quasiminimal degree in the
nonuniform dense or effective dense degrees. On the other hand, if
$\mathcal E(A) \leq\sub{nd} X$ then $A$ is computable from every dense
oracle for $X$, and hence from every coarse description of $X$, which
implies that $\mathcal E(A) \leq\sub{nc} X$. Thus the fact that every
weakly $2$\-/random or $1$\-/generic set has quasiminimal degree in
the nonuniform coarse degrees implies that the same is true in the
nonuniform dense degrees, and hence in the uniform dense degrees.
Whether $1$-randomness suffices to ensure quasiminimality in the
uniform dense degrees remains open. (See Open
Question~\ref{q:quasiminimal}.)

Cholak and Igusa~\cite[Propositions 5.4 and 5.5]{CI} showed that if a
set is $1$-random or
$1$-generic, then it is quasiminimal in the cofinite degrees (with
respect to the embedding $\mathcal R$). Since
scf-reducibility implies cf-reducibility, such a set must also be
quasiminimal in the strong-cofinite degrees, and hence in the
uniform effective dense degrees.

As mentioned above, Cholak, Hirschfeldt, and Igusa~(see~\cite{CI})
showed that
$1$-generic sets are quasiminimal for nonuniform generic reducibility, as
are weakly $2$-random sets. By forcing an effective dense oracle in place
of their generic oracle, their methods also apply to nonuniform
effective dense reducibility. Thus we have the following theorem.

\begin{thm}[including results of Hirschfeldt, Jockusch, Kuyper, and
Schupp~\cite{HJKS}; Cholak and Igusa~\cite{CI}; and Cholak,
Hirschfeldt, and Igusa (see~\cite{CI})]
Every weakly $2$-random or $1$-generic set is quasiminimal in all of
the asymptotic degree structures considered here. Every $1$-random
set is quasiminimal in the uniform generic, coarse, and effective
dense degrees, but there are $1$-random sets that are not
quasiminimal in the nonuniform generic, coarse, dense, and effective
dense degrees.
\end{thm}

This theorem leaves open only the aforementioned question of whether
there are $1$-random sets that are not quasiminimal in the uniform
dense degrees.

It should also be noted that, as shown by Igusa (see~\cite[Theorem 4.3]{HJKS}), if
$\gamma(X)=1$ (so in particular if $X$ is densely computable) but $X$
is not coarsely computable, then $X$ has quasiminimal degree in both
the uniform and nonuniform coarse degrees.

\section{Measure, upper cones, and minimal pairs}
\label{sec:upperCones}

We now consider the sizes of upper cones and existence of minimal
pairs in settings arising from notions of asymptotic computability. A
reason for studying these issues in conjunction is given by the work
of Hirschfeldt, Jockusch, Kuyper, and Schupp~\cite{HJKS}. They showed
in \cite[Theorem 5.2]{HJKS} that if $A \in 2^\omega$ is not coarsely computable, then the class of
$X \in 2^\omega$ such that $A$ is coarsely computable relative to $X$
has measure $0$, and indeed does not contain any set that is weakly
$3$\-/random relative to $A$. They then used this result to show that
if $X$ is not coarsely computable and $Y$ is weakly $3$\-/random
relative to $X$, then any set that is coarsely computable relative
both to $X$ and to $Y$ must be coarsely computable. In this case, we
say that $X$ and $Y$ form a \emph{minimal pair for relative coarse
computability} (and similarly for our other notions of asymptotic
computation). It follows that $X$ and $Y$ also form a minimal pair in
the (uniform or nonuniform) coarse degrees of sets.

Earlier, Downey, Jockusch, and Schupp~\cite[Section 7]{DJS} asked
whether there are minimal pairs in the uniform generic degrees. This
question is still open, in both the uniform and nonuniform cases (see
Open Question~\ref{q:minpairs}), but a partial answer was given by
Igusa~\cite[Theorem 2.1]{I1}, who showed that there are no minimal
pairs for relative generic computability.

In this section, we extend the upper cone result in~\cite[Theorem 5.2]{HJKS} to
all of our notions of asymptotic computation. For dense computability
(but not for generic or effective dense computability), we will then
be able to adapt the method in~\cite{HJKS} to prove the existence of
minimal pairs, though at a higher level of randomness.

We begin with a technical result that provides a modified version of
Fubini's Theorem constraining the relation between asymptotic density
and Lebesgue measure on Cantor space. For a sequence 
$\mathcal{S}=\{\mathcal{S}_n\}_{n\in\omega}$
 of Lebesgue-measurable subsets of
$2^\omega$ and a set $A \subseteq \omega$,
let $\mathcal S(A)= \{n : A \in \mathcal S_n\}$. Let $\mu$ denote
Lebesgue measure on Cantor space.

\begin{lem}
\label{lem:asymptoticFubini}
Let $a$, $b$, $q$, and $\mathcal S =\{\mathcal{S}_n\subseteq
2^{\omega}\}_{n\in\omega}$ be such that
\[
\overline{\rho}(\{n : \mu(\mathcal{S}_n)<q\})>a.
\]
and
\[
\mu(\{A : \rho(\mathcal S(A))=1\}) > b.
\]
Then $(1-q)a+b \leq 1$.
\end{lem}

\begin{proof}
Fix $r<1$ and let
\[
\mathcal{X}_n=\{A : (\forall k>n)\, \rho_k(\mathcal S(A))>r\}.
\]
Since the union of these classes contains $\{A : \rho(\mathcal
S(A))=1\}$, it must have measure greater than $b$; since the classes
are nested, there must be some finite $N$ such that
$\mu(\mathcal{X}_n)>b$ for all $n>N$.

As the set of $j$ such that $\mu(\mathcal{S}_j)<q$ has upper density
greater than $a$, there must be some $n>N$ such that
$\mu(\mathcal{S}_j)<q$ for at least $an$ many $j<n$.

As integration commutes with finite sums,
\begin{equation}
\label{6.1eq}
\frac{1}{n}\sum_{j<n}\int_{2^{\omega}}{\mathds{1}_{\mathcal{S}_j}}d\mu=\int_{2^{\omega}}{\frac{1}{n}\sum_{j<n}{\mathds{1}_{\mathcal{S}_j}}}d\mu,
\end{equation}
where $\mathds{1}_{\mathcal{S}_j} : 2^\omega \to \{0, 1\}$ is $1$ on
$S_j$ and $0$ elsewhere.
Furthermore,
\[
\frac{1}{n}\sum_{j<n}\int_{2^{\omega}}{\mathds{1}_{\mathcal{S}_j}}d\mu
=\frac{1}{n}\sum_{j<n} \mu(\mathcal{S}_j).
\]
By our choice of $n$, there are at least $an$ many $j<n$ for which $\mu(\mathcal{S}_j)<q$; for the remaining $S_j$, we can still bound their measure by $\mu(\mathcal{S}_j)\leq 1$. Therefore,
\[
\frac{1}{n}\sum_{j<n} \mu(\mathcal{S}_j)<\frac{1}{n}\left(1(n-an)+q(an)\right)=(1-a)+qa=1-(1-q)a.
\]
On the other hand, considering the right-hand side of equation~(\ref{6.1eq}), we have
\[
\int_{2^{\omega}}{\frac{1}{n}\sum_{j<n}{\mathds{1}_{\mathcal{S}_j}}}d\mu
=\int_{A \in 2^{\omega}}{\rho_n(\mathcal S(A))}d\mu
\geq \int_{A \in \mathcal{X}_n}{\rho_n(\mathcal S(A))} d\mu
>\mu(\mathcal{X}_n)r >br.
\]

Therefore, $1-(1-q)a>br$ and hence $(1-q)a+br<1$. Since $r<1$ is
arbitrary, $(1-q)a+b\leq 1$.
\end{proof}

The following statement puts the above result into a more useful form
for our purposes.

\begin{prop}
\label{prop:majorityVoting}
Let $\mathcal S=\{\mathcal{S}_n\subseteq 2^{\omega}\}_{n\in\omega}$
and $q$ be such that
\[
\mu(\{A\in 2^{\omega} : \rho(\mathcal S(A))=1\})>q.
\]
Then
\[
\rho(\{ n : \mu(\mathcal{S}_n)\geq q\})=1.
\]
\end{prop}

\begin{proof}
If $q=0$ then the statement is trivial, so we may assume that $q>0$.
Let $p=\overline{\rho}(\{n : \mu(\mathcal{S}_n)<q\})$ and assume for a
contradiction that $p>0$. Let $\epsilon>0$ be such that
$\frac{1}{p\epsilon} \in \mathbb N$ and
\[
\mu(\{A\in 2^{\omega} : \rho(\mathcal S(A))=1\})>q+\epsilon.
\]

Let $I=\{n_0<n_1<\cdots\}$ be the sequence of all $n$ such that either
$\mu(\mathcal S_n)<q$ or $n \equiv 0 \pmod
{\frac{1}{p\epsilon}}$. Then $I$ has lower density at least
$p\epsilon$ and upper density at most $p(1+\epsilon)$. We now pass to
the subsequence $\mathcal T=\{\mathcal S_{n_j}\}_{j \in
\omega}$. Since $I$ has positive lower density, every set of density
$1$ must intersect $I$ with density $1$ within $I$, so
\[
\mu(\{A\in 2^{\omega} : \rho(\mathcal T(A))=1\})>q+\epsilon.
\]
Since the upper density of $I$ is at most $p(1+\epsilon)$, given any
$\delta>0$, there is some $N$ so that for all $n>N$, we have
$\rho_n(I)<p(1+\epsilon)+\delta$. However, $I$
contains all $n$ such that $\mu(\mathcal S_n)<q$; these form a set $G$ of
upper density $p$ by assumption, and we have $\rho_n(G)  > p - \delta$ for
infinitely many $n$.    We therefore see that
\[
\overline{\rho}(\{j : \mu(\mathcal{S}_{n_j})<q\})
\geq \frac{p-\delta}{p(1+\epsilon)+\delta} =\frac{1-\frac{\delta}{p}}{1+\epsilon+\frac{\delta}{p}}.
\]
Since this inequality holds for all $\delta>0$, we must have
\[
\overline{\rho}(\{j : \mu(\mathcal{S}_{n_j})<q\}) \geq \frac{1}{1+\epsilon}
>1-\epsilon.
\]

Let $a=1-\epsilon$ and $b=q+\epsilon$. Then
$(1-q)a+b=(1-q)(1-\epsilon)+q+\epsilon=1+q\epsilon>1$, and $\mathcal
T$ has the properties in the statement of
Lemma~\ref{lem:asymptoticFubini}, yielding the desired contradiction.
\end{proof}

This result is the foundation for our arguments in the remainder of
this section, as it allows us to use majority-vote arguments to construct
computable asymptotic descriptions of sets with positive-measure upper
cones. We first apply this technique to generic computability.

\begin{thm}
\label{thm:genericUpperCones}
If $f : \omega \to \omega$ is not generically computable then
\[
\mu(\{X : f \textrm{ is generically computable relative to } X\})=0.
\]
Indeed, no element of this set can be weakly $4$\-/random relative to
$f$. Thus all nontrivial upper cones in the (uniform or nonuniform)
generic degrees of sets have measure $0$.
\end{thm}

\begin{proof}
For a functional $\Phi$, let
\[
\mathcal{F}_{\Phi}=\{X : \Phi^X\text{ is a generic description of } f\}.
\]
Then $\mathcal{F}_{\Phi}$ is a $\Pi^{0,f}_4$\-/class, as $X \in \mathcal
F_\Phi$ if and only if
\begin{enumerate}[\rm 1.]

\item for every $n$ and $s$, if $\Phi^X(n)[s]\converges$ then
$\Phi^X(n)=f(n)$ and

\item for every $k$ there is an $N$ such that for every $n>N$ there is
an $s$ for which $\rho_n(\dom \Phi^X[s]) > 1-\frac{1}{k}$.
  
\end{enumerate}
Thus it suffices to fix $\Phi$ and show that
$\mu(\mathcal{F}_{\Phi})=0$.

Assume for a contradiction that $\mu(\mathcal{F}_{\Phi})>0$. By
Lebesgue Density (see e.g.~\cite[Theorem 1.2.3]{DH}), we may assume that
$\mu(\mathcal{F}_{\Phi})>\frac{2}{3}$. Let $\mathcal{S}_n=\{X :
\Phi^X(n) = f(n)\}$ and let $\mathcal S=\{\mathcal S_n\}_{n
\in \omega}$. Since $\mu(\{X : \rho(\mathcal S(X))=1\}) \geq
\mu(\mathcal{F}_{\Phi})>\frac{2}{3}$,
Proposition~\ref{prop:majorityVoting} implies that
$\mu(\mathcal{S}_n)\geq\frac{2}{3}>\frac{1}{2}$ for density\-/$1$ many
$n$. We use
this fact to compute a generic description of $f$, using a
majority-vote scheme.

For each $n$, we wait until we see a class of measure greater than
$\frac{1}{2}$ such that for some $i$ and each $X$ in this class, we
have $\Phi^X(n) = i$, and define $d(n)=i$. (Note that for
each $n$ there is at most one such $i$.) Then $d$ is a partial
computable function, and $d(n)$ can never converge to a value other
than $f(n)$, as otherwise we would have a class of measure greater than
$\frac{1}{2}$ disjoint from $\mathcal{F}_{\Phi}$. Furthermore, for
each of the density\-/$1$ many $n$ such that
$\mu(\mathcal{S}_n)>\frac{1}{2}$, we have that
$d(n)\converges$. Thus $d$ is a computable generic description of $f$.
\end{proof}

A similar proof gives us the same result for dense computability.

\begin{thm}
\label{thm:denseUpperCones}
If $f : \omega \to \omega$ is not densely computable then
\[
\mu(\{X : f \textrm{ is densely computable relative to } X\})=0.
\]
Indeed, no element of this set can be weakly $4$\-/random relative to
$f$. Thus all nontrivial upper cones in the (uniform or nonuniform)
dense degrees of sets have measure $0$.
\end{thm}

\begin{proof}
For a $\{0,1\}$\-/valued functional $\Phi$, let
\[
\mathcal{F}_{\Phi}=\{X : \Phi^X\text{ is a dense description of } f\}.
\]
We again have that $\mathcal{F}_{\Phi}$ is a $\Pi^{0,f}_4$\-/class, as
$X \in \mathcal F_\Phi$ if and only if for every $k$ there is an $N$ such that
for every $n>N$ there is an $s$ for which $\rho_n(\{m :
\Phi^X[s](m) = f(m)\}) > 1-\frac{1}{k}$.

Assume for a contradiction that $\mu(\mathcal{F}_{\Phi})>0$. We again
may assume that $\mu(\mathcal{F}_{\Phi})>\frac{2}{3}$, define
$\mathcal S_n$ as before, and conclude that
$\mu(\mathcal{S}_n)>\frac{1}{2}$ for density\-/$1$ many $n$. We then
define the partial computable function $d$ as before. If
$d(n)\diverges$ or $d(n) \converges \neq f(n)$, then the class of all
$X$ such that $\Phi^X(n)\diverges$ or $\Phi^X(n) \converges \neq f(n)$
has measure greater than $\frac{1}{2}$, which implies that
$\mu(\mathcal{S}_n)<\frac{1}{2}$, so there can be only density\-/$0$
many such $n$. Thus $d$ is a computable dense description of $f$.
\end{proof}

In the case of effective dense computability, we can again use a
similar argument, but the level of randomness needed is lower (as in
the case of coarse computability).

\begin{thm}
\label{thm:effectiveUpperCones}
If $f : \omega \to \omega $ is not effectively densely computable then
\[
\mu(\{X : f \textrm{ is effectively densely computable relative to }
X\})=0.
\]
Indeed, no element of this set can be weakly $3$\-/random relative to
$f$. Thus all nontrivial upper cones in the (uniform or nonuniform)
effective dense degrees of sets have measure $0$.
\end{thm}

\begin{proof}
For an $(\omega \cup \{\square\})$\-/valued functional $\Phi$, let
\[
\mathcal{F}_{\Phi}=\{X : \Phi^X\text{ is an effective dense
description of } f\}.
\]
In this case, $\mathcal{F}_{\Phi}$ is a $\Pi^{0,f}_3$\-/class, as
$X \in \mathcal F_\Phi$ if and only if 
\begin{enumerate}[\rm 1.]

\item for every $n$ there is an $s$ such that $\Phi^X(n)[s]\converges
\in \{f(n),\square\}$ and

\item for every $k$ there is an $N$ such that for every $n>N$ and
every $s$, if $\Phi^X(m)[s]\converges$ for all $m<n$, then  $\rho_n(\{m :
\Phi^X[s](m)=\square\}) < \frac{1}{k}$.
  
\end{enumerate}

Assume for a contradiction that $\mu(\mathcal F_{\Phi})>0$. We may
then assume that $\mu(\mathcal{F}_{\Phi})>\frac{2}{3}$, define
$\mathcal S_n$ as before, and conclude that
$\mu(\mathcal{S}_n)\geq\frac{2}{3}$ for density\-/$1$
many $n$. For each $n$, we wait until we see a class of measure greater than
$\frac{1}{3}$ such that for some $i \in \omega \cup \{\square\}$ and
each $X$ in this class, we have $\Phi^X(n) = i$, and define
$d(n)=i$. (There may be more than one $i$ for which such a class
exists, but we choose the first one we find.) If $X \in \mathcal
F_\Phi$, then $\Phi^X$ is total and $\Phi^X(n) \in \{f(n),\square\}$
for all $n$, so for each $n$, there are either more than
measure-$\frac{1}{3}$ many $X$ such that $\Phi^X(n) = f(n)$ or more than
measure-$\frac{1}{3}$ many $X$ such that $\Phi^X(n) = \square$. Thus $d$
is a total computable function and $d(n) \in \{f(n),\square\}$ for
all $n$. For each of the density\-/$1$ many $n$ such
that $\mu(\mathcal{S}_n) \geq \frac{2}{3}$, there are at most
measure\-/$\frac{1}{3}$ many $X$ such that $\Phi^X(n)=\square$, so $d(n)
\neq \square$. Thus $d$ is a computable effective dense description of
$f$.
\end{proof}

We can also use a similar argument to give a different proof of the
analogous result for coarse computability proved in~\cite[Theorem 5.2]{HJKS}.
This argument is simpler than the original proof, given
Proposition~\ref{prop:majorityVoting}, and also applies to functions.

\begin{thm}[Hirschfeldt, Jockusch, Kuyper, and Schupp~\cite{HJKS}]
\label{thm:hjks}
If $f: \omega \to \omega$ is not coarsely computable then
\[
\mu(\{X : f \textrm{ is coarsely computable relative to } X\})=0.
\]
Indeed, no element of this set can be weakly $3$\-/random relative to
$f$. Thus all nontrivial upper cones in the (uniform or nonuniform)
coarse degrees of sets have measure $0$.
\end{thm}

\begin{proof}
For a functional $\Phi$, let
\[
\mathcal{F}_{\Phi}=\{X : \Phi^X\text{ is a coarse description of }
f\}.
\]
As noted in~\cite[proof of Theorem 5.2]{HJKS}, $\mathcal{F}_{\Phi}$ is a
$\Pi^{0,f}_3$\-/class. (The argument there was given for the case of a
$\{0,1\}$\-/valued $f$, but clearly applies in general.) Assume for a
contradiction that
$\mu(\mathcal{F}_{\Phi})>0$. We may assume that
$\mu(\mathcal{F}_{\Phi})>\frac{3}{4}$, define $\mathcal S_n$ as
before, and conclude that $\mu(\mathcal{S}_n)\geq\frac{3}{4}$ for
density\-/$1$ many $n$. Define $d$ as follows. For each $n$, wait
until we see a class $\mathcal C$ of measure at least $\frac{3}{4}$
such that 
$\Phi^X(n)\converges$ for all $X \in \mathcal C$. If there is an $i$
such that $\Phi^X(n)=i$ for
at least measure-$\frac{1}{2}$ many $X \in \mathcal C$, then let
$d(n)=i$. Otherwise, let $d(n)=0$. Note that, if $d(n)\converges \neq
f(n)$, then
$\mu(\mathcal C \cap \mathcal S_n)<\frac{1}{2}$, so $\mu(\mathcal S_n)
< \frac{3}{4}$.

Since $\Phi^X$ is total for more than measure\-/$\frac{3}{4}$ many $X$, the
function $d$ is a total computable function. If $d(n) \neq f(n)$ then, as noted
above, $\mu(\mathcal{S}_n) < \frac{3}{4}$, so there can be only density\-/$0$
many such $n$. Thus $d$ is a computable coarse description of $f$.
\end{proof}

It was shown in~\cite[Theorem 5.6]{HJKS} that the weak $3$\-/randomness condition in
this theorem cannot in general be brought down to $2$\-/randomness, but
we do not know whether the levels of randomness needed for the other
theorems in this section can be decreased. (See Open
Question~\ref{q:randlevel}.)

As mentioned above, Theorem~\ref{thm:hjks} has as a consequence that
if $X$ is not coarsely computable and $Y$ is weakly $3$\-/random
relative to $X$, then $X$ and $Y$ form a minimal pair for relative
coarse computability, and hence for both uniform and nonuniform coarse
reducibility. A variant on the proof of Corollary 5.3 in~\cite{HJKS}
allows us to establish an analogous result for dense
computability. The coarse computability proof is based on the
observation that any two coarse descriptions of the same set are
necessarily coarse descriptions of each other. The analogous statement
makes little sense for dense descriptions, as these are partial
functions. However, by careful choice of a completion of a dense
description, we can recover the construction of minimal pairs as a
consequence of Theorem~\ref{thm:denseUpperCones} without increasing
the level of randomness involved.

\begin{cor}
\label{cor:minimal}
If $Y$ is not densely computable and $X$ is weakly $4$\-/random relative
to $Y$, then $X$ and $Y$ form a minimal pair for relative dense
computability (and hence for both uniform and nonuniform dense
reducibility).
\end{cor}

\begin{proof}
If $f$ is a partial computable function with $\rho(\dom f)=1$, then,
as noted in the proof of Proposition~\ref{prop:abgd}, $\dom f$ has a
computable subset $S$ of positive lower density. Since $X$ is $1$\-/random,
the density within $S$ of numbers $n$ such that $f(n) \neq X(n)$ must
be $\frac{1}{2}$, and hence the overall density of such $n$ is
positive. Thus $X$ is not densely computable.

Now suppose that $C$ is densely computable relative both to $X$ and to
$Y$. Fix dense descriptions $\Phi^X$ and $\Psi^Y$ of $C$. Let $P$ be a
set that is both PA over $Y$ and low relative to $Y$. Then $P$
computes a $\{0,1\}$\-/valued completion $D$ of $\Psi^Y$. Since $D$ agrees
with $C$ on a set of density $1$, we have that $\Phi^X$ is a dense
description of $D$, and hence $D$ is densely computable relative to
$X$. However, $D$ is $P$\-/computable, and hence is low relative to $Y$,
so $X$ is weakly $4$\-/random relative to $D$. By
Theorem~\ref{thm:denseUpperCones}, $D$ must be densely
computable. Since any dense description of $D$ is also a dense
description of $C$, it follows that $C$ is densely computable.
\end{proof}

By the aforementioned result of Igusa~\cite{I1}, this method of
constructing minimal pairs cannot work for generic computability. It
also does not seem to work for effective dense computability, at least
not in any straightforward way. (See Open Question~\ref{q:minpairs}.) It is
worth noting, however, that if $Y$ is not generically computable, $X$
is weakly $4$\-/random relative to $Y$, and $C$ is generically
computable relative to both $X$ and $Y$ then, by
Corollary~\ref{cor:minimal}, $C$ is densely computable. In other
words, there are pairs of sets $X,Y$ such that all witnesses to the
fact that $X$ and $Y$ do not form a minimal pair for relative generic
computability are in a sense ``close to generically computable''.

We can also consider the sizes of upper cones in the sense of
category. Here we know less than in the case of measure, but do have
the following results.

\begin{thm}
\label{thm:onegeneric}
If $f$ is not generically computable and $X$ is $1$\-/generic relative
to $f$ then $X$ does not compute a generic description of $f$. Thus
all nontrivial upper cones in the (uniform or nonuniform) generic
degrees are meager.
\end{thm}

\begin{proof}
Suppose that $X$ is $1$\-/generic relative to $f$ and $\Phi^X$ is a
generic description of $f$. The set of $\sigma \in 2^{<\omega}$ such
that $\Phi^\sigma(n) \converges \neq f(n)$ for some $n$ is
$f$\-/c.e. Since $X$ does not meet this set, it must avoid it. Thus
there is a $\tau \prec X$ such that for all $\sigma \succ \tau$, if
$\Phi^\sigma(n)\converges$ then $\Phi^\sigma(n)=f(n)$. Define a
partial computable function $d$ by searching for strings $\sigma \succ
\tau$ and numbers $n$ such that $\Phi^\sigma(n)\converges$ and letting
$d(n)=\Phi^\sigma(n)$. Then $d$ is a partial description of $f$. Since
$\tau \prec X$, we have $\dom \Phi^X \subseteq \dom d$, so $\rho(\dom
d)=1$. Thus $d$ is a computable generic description of $f$.
\end{proof}

For the next result, we will use the following lemma
from~\cite{HJMS}. Let $J_k$ be the interval $[2^k-1,2^{k+1}-1)$. For a
set $C$, let $d_k(C)$ be the density of $C$ on $J_k$, that is,
$\frac{|C \cap J_k|}{2^k}$. Let $\overline{d}(C) = \limsup_k d_k(C)$.

\begin{lem}[Hirschfeldt, Jockusch, McNicholl, and Schupp~{\cite[Lemma
5.9]{HJMS}}]
\label{lem:hjms}  
For every set $C$, we have $\frac{\overline{d}(C)}{2} \leq
\overline{\rho}(C) \leq 2\overline{d}(C)$.
\end{lem}

Let $\gamma^X$ be the relativization of the coarse computability bound
$\gamma$ to the set $X$.
  
\begin{thm}
If $\gamma(f)<1$ and $X$ is $1$\-/generic relative to $f$ then
$\gamma^X(f)<1$.
\end{thm}

\begin{proof}
Suppose $X$ is $1$\-/generic relative to $f$ and $\gamma^X(f)=1$. Let
$\epsilon>0$ be rational. Let $\Phi$ be a functional such that
$\Phi^X$ is total and $\overline{\rho}(\{n : \Phi^X(n) \neq f(n)\}) <
\epsilon$. Then there is an $M$ such that if $m>M$ then $\rho_m(\{n :
\Phi^X(n) \neq f(n)\}) < \epsilon$. Let $S$ be the set of all $\sigma
\in 2 ^{<\omega}$ for which there is an $m>M$ such that
$\Phi^\sigma(n)\converges$ for all $n<m$ and $\rho_m(\{n :
\Phi^\sigma(n)\converges \neq f(n)\}) \geq \epsilon$. Then $S$ is
$f$\-/c.e., and $X$ does not meet $S$, so $X$ must avoid $S$. That is,
there is a $\tau \prec X$ such that for all $\sigma \succ \tau$, we
have $\sigma \notin S$.

Now define a computable function $g$ as follows.  Let $J_k$ be as
above. For each $k$ in turn, search for a $\sigma \succ \tau$ such
that $\Phi^\sigma(n)\converges$ for all $n<2^{k+1}$. Such a $\sigma$
must exist because $\tau \prec X$ and $\Phi^X$ is total. Once $\sigma$
is found, let $g(n)=\Phi^\sigma(n)$ for all $n \in J_k$. If $2^{k+1}
\geq M$ then, by the choice of $\tau$, we have $\rho_{2^{k+1}+1}(\{n :
\Phi^\sigma(n) \neq f(n)\}) < \epsilon$. Since $|J_k| = 2^k$, it
follows that $d_k(\{n : g(n) \neq f(n)\}) = d_k(\{n : \Phi^\sigma(n)
\neq f(n)\}) < 2\epsilon$. Thus $\overline{d}(\{n : g(n) \neq f(n)\})
\leq 2\epsilon$, so by Lemma~\ref{lem:hjms}, $\overline{\rho}(\{n :
g(n) \neq f(n)\}) \leq 4\epsilon$, and hence $\gamma(f) \geq
1-4\epsilon$. Since $\epsilon$ is arbitrary, $\gamma(f)=1$.
\end{proof}

\section{Open questions}
\label{sec:questions}

In this final section, we gather some questions left open in this
paper.

\begin{oq}
\label{q:degrees}
For each reducibility r among the ones defined in
Section~\ref{sec:relativization}, is it the case that every function
is r\-/equivalent to a set?
\end{oq}

\begin{oq}
\label{q:reds}
The following questions can be asked in both uniform and nonuniform
versions: Does effective dense reducibility imply generic
reducibility?  Does effective dense reducibility imply coarse
reducibility?  Does generic reducibility imply dense reducibility?
Does coarse reducibility imply dense reducibility?  Does effective
dense reducibility imply dense reducibility?

More generally, how do effective dense reducibility and dense
reducibility fit into the picture of implications between notions of
robust information coding given by Dzhafarov and Igusa~\cite{DI}?
\end{oq}

\begin{oq}
\label{q:minpairs}
Are there minimal pairs in the (uniform or nonuniform) generic or
effective dense degrees? Are there minimal pairs for relative
effective dense computability?
\end{oq}

\begin{oq}
\label{q:quasiminimal}
Do all $1$-random sets have quasiminimal degree in the uniform dense degrees?
\end{oq}

\begin{oq}
\label{q:randlevel}
Can the weak $4$\-/randomness condition in
Theorems~\ref{thm:genericUpperCones} and \ref{thm:denseUpperCones},
and Corollary~\ref{cor:minimal}, or the weak $3$\-/randomness condition
in Theorem~\ref{thm:effectiveUpperCones} be improved?
\end{oq}

\begin{oq}
Are there analogs of Theorem~\ref{thm:onegeneric} for coarse, dense,
and effective dense computability? In particular, are nontrivial upper
cones in the coarse, dense, and effective dense degrees meager?
\end{oq}

\begin{oq}
\label{q:dense}
If $h \leq\sub{nd} g$, does every dense oracle for $g$ compute a dense
oracle for $h$? If $h \leq\sub{ud} g$, does every dense oracle for $g$
uniformly compute a dense oracle for $h$?

A \emph{weak cofinite oracle} for a function $g$ is a cofinite oracle
for a mod-finite description of $g$. If $h \leq\sub{wcf} g$, does
every weak cofinite oracle for $g$ uniformly compute a weak cofinite
oracle for $h$? 
\end{oq}

\begin{oq}
\label{q:cf}
The following questions can be asked in Cantor space or in Baire
space: Do either of mf\-/reducibility and scf\-/reducibility imply the
other?  Do either of cf\-/reducibility and wcf\-/reducibility imply
the other? Does ubfb\-/reducibility imply scf\-/reducibility? Does
wcf\-/reducibility imply ubfb\-/reducibility?

Does mf\-/reducibility imply wcf\-/reducibility in Baire space?
\end{oq}

\end{document}